\documentclass[10pt,a4paper]{my_siamltex}
\usepackage{amsmath,amssymb}
\usepackage{ulem}
\usepackage{tikz}
\usetikzlibrary{patterns}
\definecolor{bfonce}{rgb}{0.,0.,0.8}	
\definecolor{bclair}{rgb}{0.87,0.92,1.}
\definecolor{orangec}{rgb}{1.,0.6,0.}

\newcommand{\edge}{\sigma}
\newcommand{\edges}{{\mathcal E}}
\newcommand{\edgesint}{{\mathcal E}_{{\rm int}}}
\newcommand{\edgesext}{{\mathcal E}_{{\rm ext}}}
\newcommand{\mesh}{{\mathcal M}}
\newcommand{\tdisc}{{\mathcal T}}
\newcommand{\bfb}{{\boldsymbol b}}
\newcommand{\bfF}{{\boldsymbol F}}
\newcommand{\bfn}{{\boldsymbol n}}
\newcommand{\bfx}{{\boldsymbol x}}
\newcommand{\bfy}{{\boldsymbol y}}
\newcommand{\bfpsi}{{\boldsymbol \psi}}
\newcommand{\charac}{1 \!\!1}
\newcommand{\gradi}{{\boldsymbol \nabla}}
\newcommand{\dive}{{\rm div}}
\newcommand{\dt}{\,\mathrm{d}t}
\newcommand{\dx}{\,\mathrm{d}\bfx}
\newcommand{\dy}{\,\mathrm{d}\bfy}
\newcommand{\xN}{\mathbb{N}}
\newcommand{\xR}{\mathbb{R}}
\newcommand{\exm}{^{(m)}}
\newcommand{\ie}{{\it i.e.}}

\newcommand{\bli}{\begin{list}{-}{\itemsep=1ex \topsep=1ex \leftmargin=0.5cm \labelwidth=0.3cm \labelsep=0.2cm \itemindent=0.cm}}
\newtheorem{remark}{Remark}[section]
\title{On the weak consistency of finite volumes schemes for conservation laws on general meshes}
\author{
T. Gallou\"et \thanks{Universit\'e d'Aix-Marseille ({\tt thierry.gallouet@univ-amu.fr})} \and
R. Herbin \thanks{Universit\'e d'Aix-Marseille ({\tt raphaele.herbin@univ-amu.fr})} \and
J.-C. Latch\'e \thanks{Institut de Radioprotection et de S\^{u}ret\'{e} Nucl\'{e}aire (IRSN), PSN-RES/SA2I, Cadarache, St-Paul-lez-Durance, 13115, France ({\tt jean-claude.latche@irsn.fr})}
}
\begin{document}
\maketitle
\begin{abstract}
The aim of this paper is to develop some tools in order to obtain the weak consistency of (in other words, an analogue of the Lax-Wendroff theorem for) finite volume schemes for balance laws in the multi-dimensional case and under minimal regularity assumptions for the mesh.
As in the seminal Lax-Wendroff paper, our approach relies on a discrete integration by parts of the weak formulation of the scheme. Doing so, a discrete gradient of the test function  appears; 
the central argument for the scheme consistency is to remark that this discrete gradient is convergent in $L^\infty$ weak $\star$.
\end{abstract}
%
%
\section{Introduction}

In the early sixties, P.D. Lax and B. Wendroff established that, on uniform 1D grids, a flux-consistent and conservative cell-centered finite-volume scheme for a conservation law is weakly consistent, in the sense that the limit of any a.e. convergent sequence of $L^\infty$-bounded numerical solutions, obtained with a sequence of grids with mesh and time steps tending to zero, is a weak solution  of the conservation law \cite{lax-60-sys}; this result is known as the Lax-Wendroff theorem, and is reported in many textbooks with some variants:  see e.g. \cite[Section 12.10]{lev-02-fin} with a BV bound assumption on the scheme, and \cite[Theorem 21.2]{egh-00-fin} for a generalisation to non-uniform meshes.
However, convergence proofs on unstructured meshes which were obtained for nonlinear scalar conservation laws in the 90's do not use the Lax-Wendroff theorem; indeed, finite volume schemes on unstructured meshes are known to be in general not TVD (see e.g. an example in \cite{cha-92-wea}), so that a compactness property in $L^1_{loc}$ is not easy to obtain, although it does hold in fact but results from the proof of uniqueness of the so-called entropy process solution, see e.g. \cite[chapter VI]{egh-00-fin} and references therein.

Nevertheless, from a practical point of view, the Lax-Wendroff theorem, even if weaker than a full convergence proof, may be fundamental for the design of numerical schemes.
In particular, for many hyperbolic systems (especially in the multi-dimensional case), it represents the essential part of the theoretical foundations, since provable estimates on numerical solutions are too weak to provide sufficient compactness to undertake any convergence study.    

\medskip
The seminal Lax-Wendroff paper has been the starting point for several research works, aimed at relaxing the original assumptions: for instance, an extension to schemes which do not admit a local conservative formulation may be found in \cite{car-03-cons,abg-03-hig,abg-06-res,shi-18-loc}; the derivation of analogue consistency results for non-conservative hyperbolic systems is presented in \cite{fed-00-sim,mun-11-conv}.
In the multi-dimensional case and for standard finite volumes schemes, Lax-Wendroff type results have been progressively extended to general (and, in particular, unstructured) discretizations  \cite{kro-96-lax}, in \cite[Section 4.2.2]{god-96-num} for two-dimensional simplicial meshes and in \cite{ell-07-lax} for multi-dimensional meshes.

\medskip
The present paper follows the same route, in the sense that we still extend the application range of the Lax-Wendroff theorem in terms on constraints on the mesh, in particular relaxing the quasi-uniformity assumption.
However, compared to \cite{ell-07-lax}, the assumptions and the technique of proof are different; we consider general meshes with only a regularity assumption linked to the definition of a discrete gradient, but require the flux function to be Lipschitz continuous or at least ``lip-diag", while in \cite{ell-07-lax} the space-time grid is assumed quasi-uniform but the flux is only required to be continuous,  and we return to a strategy close to the original work: the scheme is multiplied by a test function and integrated over space and time, and a discrete integration by parts of the convection term yields the integral of the product of the numerical flux by a discrete gradient of the test function (this latter seems to appear first in \cite{eym-00-hconv}, where its convergence properties are shown for specific meshes and norms).
Then the passage to the limit of vanishing space and time steps in the scheme requires two ingredients:
\bli
\item a convergence result for the discrete gradient; contrary to what happens in the 1D case or for Cartesian grids, this convergence is only weak, namely in $L^\infty$ weak $\star$, which is however sufficient to conclude,
\item the control of some residual terms, which basically consists in the difference between the numerical solution and a space or time translate of this latter; the difficulty here lies in the fact that the translation amplitude is (locally) mesh-dependent (for instance, the function is translated from one cell to its neighbour), so that  standard results for converging sequences of functions in $L^1$ may no longer be applied.
\end{list}
The presentation is organized as follows: after a definition of the considered space discretizations (Section \ref{sec:mesh}), we address successively the two above-mentioned issues (Section \ref{sec:grad} and \ref{sec:kolm} respectively), carefully clarifying in these two sections the regularity requirements for the mesh.
Then we show in Section \ref{sec:lw} how to use the obtained results to obtain a weak consistency result for a standard finite volume discretization of a balance law; consistency requirements for the numerical flux appear in this step. 
%
%
\section{Space discretization}\label{sec:mesh}

Let $\Omega$ be an open bounded polyhedral set of $\xR^d$, $d \ge 1$.
A polyhedral partition $\mesh$ of $\Omega$ is a finite partition of $\Omega$ such that each element $K$ of this partition is measurable and has a boundary $\partial K$ that is composed of a finite union of part of hyperplanes (the faces of $K$) denoted by $\edge$, so that $\partial K=\cup_{\edge \in \edges_K} \edge$ where $\edges_K$ is the set of the faces of $K$.
Such a polyhedral partition is called a ``mesh''.
We denote by $\edges$ the set of all the faces, namely $\edges = \cup_{K \in \mesh} \edges_K $.
If $\edge \in \edges$ is a face of this partition, then one denotes by $|\edge|$ the $(d-1)$-Lebesgue measure of $\edge$.
We denote by $\edgesint$ the set of elements $\edge$ of $\edges$ such that there exist $K$ and $L$ in $\mesh$ ($K \ne L$) such that  $\edge \in \edges_K \cap \edges_L$; such a face $\edge$ is denoted by $\edge=K|L$.
The set of faces located on the boundary of $\Omega$, \ie\ $\edges \setminus \edgesint$, is denoted by $\edgesext$.
For $K \in \mesh$ we denote by $h_K$ the diameter of $K$.
The size of the mesh $\mesh$ is $h_\mesh = \max\{h_K, K \in \mesh\}$.
For $K\in\mesh$ and $\edge \in \edges_K$, we denote by $\bfn_{K,\edge}$ the normal vector of $\edge$ outward $K$.

\medskip
We also introduce now a dual mesh, that is a new partition of $\Omega$ indexed by the elements of $\edges$, namely $\Omega=\cup_{\edge \in \edges} D_\edge$.
For $\edge=K|L$, the set $D_\edge$ is supposed to be a subset of $K \cup L$ and we define $D_{K,\edge}=D_\edge \cap K$, so $D_\edge=D_{K,\edge} \cup D_{L,\edge}$ (see Figure \ref{fig:space_disc}).

\medskip
If $A$ is a measurable set of $\xR^d$, we denote by $|A|$ the Lebesgue measure of $A$.

\begin{figure}[h!]
\begin{center}
\scalebox{0.9}{
\begin{tikzpicture} 
\fill[color=bclair,opacity=0.3] (1.,2.) -- (4.,1.) -- (5.,3.) -- (3.,4.) -- (1.5,3.5) -- (1.,2.);
\fill[color=orangec!30!white,opacity=0.3] (5.,3.) -- (6.,5.) -- (3.,6.) -- (3.,4.);
\draw[very thick, color=black] (1.,2.) -- (4.,1.) -- (5.,3.) -- (3.,4.) -- (1.5,3.5) -- (1.,2.);
\draw[very thick, color=black] (5.,3.) -- (6.,5.) -- (3.,6.) -- (3.,4.);
\draw[very thick, color=black] (1.,2.) -- (0.,1.5);
\draw[very thick, color=black] (4.,1.) -- (4.5,0.);
\draw[very thick, color=black] (6.,5.) -- (7.,5.);
\draw[very thick, color=black] (3.,6.) -- (3.5,7.); \draw[very thick, color=black] (3.,6.) -- (2.,6.);
\draw[very thick, color=black] (1.5,3.5) -- (0.5,4.);
\node at (3., 2.5){$\mathbf K$}; \node at (4.5, 4.7){$\mathbf L$};
\draw[very thick, color=black] (5.,3.) -- (3.,4.) node[midway,sloped,above]{$\edge=K|L$};
\fill[color=bclair,opacity=0.3] (10.,4.) .. controls (10.1,3.) .. (10.5,2.6) .. controls (11.,2.6) .. (12.,3.);
\fill[color=orangec!30!white,opacity=0.3] (10.,4.) .. controls (10.8,5.) .. (11.2,4.8) .. controls (11.6,4.7) .. (12.,3.);
\draw[very thick, color=black] (8.,2.) -- (11.,1.) -- (12.,3.) -- (10.,4.) -- (8.5,3.5) -- (8.,2.);
\draw[very thick, color=black] (12.,3.) -- (13.,5.) -- (10.,6.) -- (10.,4.);
\draw[very thick, color=black] (8.,2.) -- (7.,1.5);
\draw[very thick, color=black] (11.,1.) -- (11.5,0.);
\draw[very thick, color=black] (13.,5.) -- (14.,5.);
\draw[very thick, color=black] (10.,6.) -- (10.5,7.); \draw[very thick, color=black] (10.,6.) -- (9.,6.);
\draw[very thick, color=black] (8.5,3.5) -- (7.5,4.);
\draw[thick, color=bfonce] (10.3,2.1) -- (10.5,2.6);
\draw[thick, color=bfonce] (10.,4.) .. controls (10.1,3.) .. (10.5,2.6);
\draw[thick, color=bfonce] (10.5,2.6) .. controls (11.,2.6) .. (12.,3.);
\draw[thick, color=bfonce] (8.5,3.5) -- (10.3,2.1);
\draw[thick, color=bfonce] (8.,2.) .. controls (9.,2.3) .. (10.3,2.1);
\draw[thick, color=bfonce] (11.,1.) .. controls (10.8,2.) .. (10.3,2.1);
\draw[thick, color=orangec!80!black] (10.,4.) .. controls (10.8,5.) .. (11.2,4.8) .. controls (11.6,4.7) .. (12.,3.);
\draw[thick, color=orangec!80!black] (13.,5.) -- (11.2,4.8);
\draw[thick, color=orangec!80!black] (10.,6.) -- (11.2,4.8);
\draw[very thick, color=black] (12.,3.) -- (10.,4.) node[midway,sloped,above]{$D_{L,\edge}$} node[midway,sloped,below]{$D_{K,\edge}$};
\end{tikzpicture}
}
\end{center}
\caption{Mesh and associated notations.}
\label{fig:space_disc}
\end{figure}
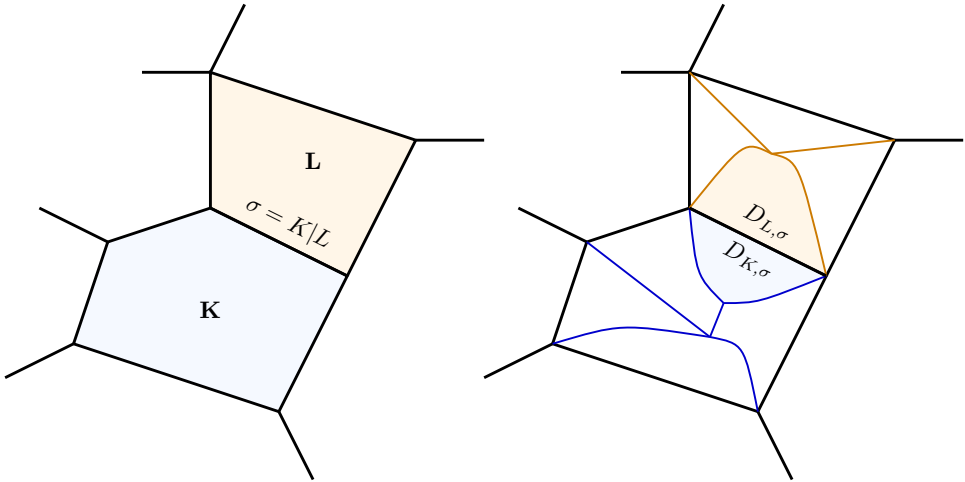
%
%
\section{A weakly convergent discrete gradient}\label{sec:grad}

Let $\varphi \in C_c^\infty(\Omega)$ and, for $K \in \mesh$, let $\bfx_K$ be a point of $K$ and $\varphi_K=\varphi(\bfx_K)$.
For $\edge \in \edges$, let
\begin{equation} \label{eq:grad}
\begin{array}{ll} \displaystyle
(\gradi_\edges \varphi)_\edge = \frac{|\edge|}{|D_\edge|}\ (\varphi_L-\varphi_K)\, \bfn_{K,\edge}
&
\mbox{ if }\edge\in\edgesint, \edge=K|L,
\\[3ex] \displaystyle
(\gradi_\edges \varphi)_\edge = 0
&
\mbox{ if }\edge\in\edgesext.
\end{array}
\end{equation}
We characterize the regularity of the mesh by the following parameter:
\begin{equation}\label{eq:deftheta1}
	\theta_\mesh^\nabla=\max_{\edge \in \edgesint,\, \edge=K|L} \frac{|\edge|\ |\bfx_L - \bfx_K|}{|D_\edge|}.
\end{equation}
 Note that for Cartesian meshes, $\theta_\mesh^\nabla=1.$ 
With this definition, we get for $\edge \in \edges$,
\begin{equation}\label{eq:grad_bound}
|(\gradi_\edges \varphi)_\edge| \leq\ \theta_\mesh^\nabla \ |\gradi \varphi|_{L^\infty(\Omega)^d},
\end{equation}
where $\gradi_\edges \varphi$  is the piecewise constant function equal to $(\gradi_\edges \varphi)_\edge$ over $D_\edge$, for $\edge \in \edges$.

\medskip
\begin{lemma}\label{lem:grad}
Let $(\mesh\exm)_{m \in \xN}$ be a sequence of meshes such that the mesh step $h_{\mesh\exm}$ tends to zero when $m$ tends to $+\infty$.
We suppose that the mesh parameters $\theta^\nabla_{\mesh\exm}$ defined by \eqref{eq:deftheta1}  are uniformly bounded with respect to $m$, {\it i.e.}:  
\begin{equation}
\exists \theta^\nabla \in \xR_+ \ : \ \forall m\in\xN \theta^\nabla_{\mesh\exm} \leq \theta^\nabla.
\label{hyp:mesh_grad}	
\end{equation}
Let $\varphi \in C_c^\infty(\Omega)$ and, for $m \in \xN$, let $\gradi_{\edges\exm} \varphi \in L^\infty(\Omega)^d$ be defined by \eqref{eq:grad}.
Then the sequence $(\gradi_{\edges\exm} \varphi)_{m \in \xN}$ is bounded in $L^\infty(\Omega)^d$ uniformly with respect to $m$ and converges to $\gradi \varphi$ in $L^\infty(\Omega)^d$ weak $\star$.
\end{lemma}

\medskip
\begin{proof}
The fact that the sequence $(\gradi_{\edges\exm} \varphi)_{m \in \xN}$ is bounded in $L^\infty(\Omega)^d$ is a straightforward consequence of Inequality \eqref{eq:grad_bound} and the assumption $\theta^\nabla_{\mesh\exm} \leq \theta^\nabla$ for $m\in\xN$.
Let $\bfpsi \in C_c^\infty(\Omega)^d$.
For $m \in \xN$ and $\edge \in \edges\exm$, let $\bfpsi_\edge$ and $\bar \bfpsi_\edge$ be defined by the mean value of $\bfpsi$ over $\edge$ and $D_\edge$ respectively.
Since $\gradi_{\edges\exm} \varphi$ is piecewise constant over the dual cells $D_\edge$, we have
\[
I\exm :=\int_\Omega \gradi_{\edges\exm}\varphi(\bfx)\, \cdot \bfpsi(\bfx) \dx=
\sum_{\edge \in \edges\exm} |D_\edge|\ (\gradi_{\edges_n}\varphi)_\edge\, \cdot\bar \bfpsi_\edge=\tilde I\exm + R\exm,
\]
with
\[
\tilde I\exm=\sum_{\edge \in \edges\exm} |D_\edge|\ (\gradi_{\edges\exm}\varphi)_\edge\, \cdot \bfpsi_\edge, \quad
R\exm=\sum_{\edge \in \edges\exm} |D_\edge|\ (\gradi_{\edges\exm}\varphi)_\edge\, \cdot (\bar \bfpsi_\edge-\bfpsi_\edge).
\]
We first consider $\tilde I\exm$.
Since $\psi$ has a compact support in $\Omega$, the quantity $\bfpsi_\edge$ vanishes for any $\edge \in \edgesext\exm$.
We thus get, using the definition \eqref{eq:grad} of the discrete gradient and reordering the sums:
\[
\tilde I\exm = \sum_{\edge \in \edgesint\exm,\, \edge=K|L} |\edge|\ (\varphi_L-\varphi_K)\, \bfn_{K,\edge} \cdot \bfpsi_\edge
=-\sum_{K\in\mesh\exm} \varphi_K \sum_{\edge \in \edges_K} |\edge|\ \bfpsi_\edge \cdot \bfn_{K,\edge}.
\]
Hence,
\[
\tilde I\exm = -\sum_{K\in\mesh\exm} \varphi(\bfx_K)\ \int_K \dive \bfpsi(\bfx) \dx
= -\int_\Omega \varphi(\bfx)\,\dive \bfpsi(\bfx) \dx +\tilde R\exm,
\]
with
\[
\tilde R\exm = \sum_{K\in\mesh\exm}\ \int_K \bigl(\varphi(\bfx)-\varphi(\bfx_K)\bigr)\ \dive \bfpsi(\bfx) \dx.
\]
The remainder term $\tilde R\exm$ may be bounded as follows
\begin{equation}\label{eq:rem1}
|\tilde R\exm| \leq |\gradi \varphi|_{L^\infty(\Omega)^d}\  |\gradi \bfpsi|_{L^\infty(\Omega)^{d\times d}}\ |\Omega|\ h_{\mesh\exm},
\end{equation}
and thus
\[
\lim_{m \to +\infty} \tilde I\exm = -\int_\Omega \varphi(\bfx)\ \dive \bfpsi(\bfx) \dx = \int_\Omega \gradi \varphi(\bfx) \cdot \bfpsi(\bfx) \dx.
\]
The term $R\exm$ reads, once again thanks to the definition \eqref{eq:grad} of the discrete gradient:
\[
R\exm = \sum_{\edge \in \edgesint\exm,\, \edge=K|L} |D_\edge|\ (\gradi_{\edges\exm}\varphi)_\edge\, \cdot (\bar \bfpsi_\edge-\bfpsi_\edge).
\]
Since $\bfpsi$ is regular, there exists $\bfx_\edge$ and $\bfx_{D_\edge}$ such that $\bfpsi_\edge=\bfpsi(\bfx_\edge)$ and $\bar \bfpsi_\edge=\bfpsi(\bfx_{D_\edge})$ respectively.
Hence, since $|\bfx_\edge-\bfx_{D_\edge}| \leq \max\,(h_K,h_L)$, we have thanks to \eqref{eq:grad_bound}:
\begin{equation}\label{eq:rem2}
|R\exm| \leq |\gradi \varphi|_{L^\infty(\Omega)^d}\  |\gradi \bfpsi|_{L^\infty(\Omega)^{d\times d}}\ |\Omega|\ \theta^\nabla_{\mesh\exm} \ h_{\mesh\exm},
\end{equation}
and thus
\[
\lim_{m \to +\infty} |R\exm| =0.
\]
The conclusion of the proof is obtained by invoking the density of the functions of $C_c^\infty(\Omega)^d$ in $L^1(\Omega)^d$.\\
\end{proof}

\medskip
\begin{remark}\label{rem:cvu}
Estimates \eqref{eq:rem1} and \eqref{eq:rem2}, show that the difference $D\exm$ defined by
\[
D\exm = \int_\Omega \bigl|\bigl(\gradi_{\edges\exm}\varphi(\bfx)- \gradi\varphi(\bfx)\bigr) \cdot \bfpsi(\bfx) \bigr| \dx
\]
 can be bounded by a parameter depending only on the sequence of meshes and on $\varphi$, $|\gradi \varphi|_{L^\infty(\Omega)^d}$ and $|\gradi \bfpsi|_{L^\infty(\Omega)^{d\times d}}$.
This point is used hereafter to extend the present convergence result to time depending functions.
\end{remark}

\medskip
\begin{remark}[Choice of $\bfx_K$ and $D_\edge$]\label{rem:D}
Note that, apart from the need to ensure the regularity of the sequence of meshes (\ie\ $\theta^\nabla_{\mesh\exm} \leq \theta^\nabla$ for $m\in\xN$), there is almost no constraint on the choice of $\bfx_K$ in $K$ and on $D_\edge$, which is just a volume associated to the face $\edge$ which, for the proof of Lemma \ref{lem:grad}, does not even need to contain $\edge$ itself. 
\end{remark}

\medskip
\begin{remark}[On the mesh regularity assumption]\label{rem:reg_grad}
Some regularity constraints for the sequence of meshes may be shown to be strong enough to control the parameter $\theta_\mesh^\nabla$.
First, let us suppose that the cells are not allowed to be too flat, \ie\ that there exists $C>0$ such that
\begin{equation}\label{eq:nflat}
|K| \geq C\,h_K^d,\quad \forall K \in \mesh\exm,\ \forall m \in \xN.
\end{equation}
Then, let us denote by $\tau_\mesh$ the parameter:
\[
\tau_\mesh = \max_{K\in\mesh,\ \edge \in \edges_K}\ \frac{|K|}{|D_{K,\edge}|}.
\]
Since the definition of $D_\edge$ is almost arbitrary, $\tau_\mesh$ may be kept bounded away from zero for any sequence of meshes, provided that the number of faces of the cells is bounded; for instance, in Section \ref{sec:lw}, we choose $D_\edge$ is such a way that $\tau_\mesh$ is equal to the inverse of the maximum number of cell faces.
We then have:
\[
\theta_\mesh^\nabla \leq \max_{\edge \in \edgesint,\, \edge=K|L} \frac{(h_K+h_L)\ \min\,(h_K^{d-1},h_L^{d-1})}{\tau_\mesh\ \max\,(|K|,|L|)}
\leq 2\, \frac C {\tau_\mesh}.
\]
The case of "flat cells" is more intricate.
However, let us suppose that, for $m\in\xN$ and $\edge \in \edgesint\exm$, $\edge=K|L$,
\bli
\item $\bfx_K$ and $\bfx_L$ may be chosen such that $K$ and $L$ are star-shaped with respect to $\bfx_K$ and $\bfx_L$ respectively,
\item we choose for $D_{K,\edge}$ and $D_{L,\edge}$ the cones of basis $\edge$ and vertex $\bfx_K$ and $\bfx_L$ respectively,
\item there exists $C>0$ such that $|(\bfx_K-\bfx_L) \cdot \bfn_{K,\edge}| \geq C\, |\bfx_K-\bfx_L|$ (which may be referred to as a "uniform orthogonality condition").
\end{list}
Then $|\bfx_K-\bfx_L|\ |\edge| \leq d \, |D_\edge|/C$ (see Fig. \ref{fig:xKxL}).
\end{remark}

\begin{figure}[h!]
\begin{center}
\scalebox{0.9}{
\begin{tikzpicture} 
\draw[thick, color=black] (-0.2,1.) -- (4.2,1.); \draw[thick, color=black] (-0.2,1.5) -- (4.2,1.5); \draw[thick, color=black] (-0.2,2.) -- (4.2,2.);
\draw[thick, color=black] (0.,0.8) -- (0.,2.2); \draw[thick, color=black] (4.,0.8) -- (4.,2.2);
\filldraw[bfonce] (2.1,1.25) circle(0.05);\node at (1.7, 1.25){$\bfx_K$};
\filldraw[bfonce] (1.9,1.75) circle(0.05); \node at (1.5,1.75){$\bfx_L$};
\draw[thin, color=bfonce] (2.1,1.25) -- (1.9,1.75);
\node at (2, 0.4){(a)};
\draw[thick, color=black] (5.8,1.) -- (10.2,1.); \draw[thick, color=black] (5.8,1.5) -- (10.2,1.5); \draw[thick, color=black] (5.8,2.) -- (10.2,2.);
\draw[thick, color=black] (6.,0.8) -- (6.,2.2); \draw[thick, color=black] (10.,0.8) -- (10.,2.2);
\filldraw[bfonce] (6.1,1.25) circle(0.05);\node at (6.5, 1.14){$\bfx_K$};
\filldraw[bfonce] (9.9,1.75) circle(0.05); \node at (9.5,1.84){$\bfx_L$};
\draw[thin, color=bfonce] (6.1,1.25) -- (9.9,1.75);
\node at (8, 0.4){(b)}; 
\end{tikzpicture}
}
\end{center}
\caption{Choice of $\bfx_K$ and $\bfx_L$ for "flat cells": (a) convenient choice, (b) quasi-orthogonality is lost when cells become flatter and flatter.}
\label{fig:xKxL}
\end{figure}
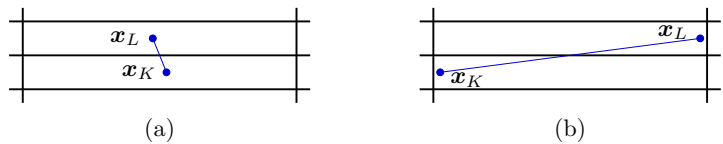

\medskip
The analysis of finite volume schemes for (systems of) conservation laws necessitates the extension of Lemma \ref{lem:grad} to time-dependent functions.
Let $T>0$ and a time discretization of the interval $(0,T)$, \ie\ a sequence $\tdisc= (t_n)_{0\leq n \leq N}$ with $0=t_0< t_1 \dots <t_n < t_{n+1} \dots < t_N=T$, be given; we define $\delta t_n=t_{n+1}-t_n$ and $\delta t_\tdisc = \max \, \{ \delta t_n,\ 0\leq n < N \}$. 
Let $\varphi \in C_c^\infty(\Omega \times [0,T])$;  then the piecewise function $\gradi_{\edges,\tdisc}\, \varphi$ is defined by:
\begin{multline} \label{eq:gradT}
\gradi_{\edges,\tdisc}\, \varphi = \sum_{n=0}^{N-1} \sum_{\sigma \in \edges} (\gradi_{\edges,\tdisc}\, \varphi)_\edge^n\ \charac_{D_\sigma}(\bfx)\ \charac_{[t_n,t_{n+1}[}(t), \mbox{ with }
\\
(\gradi_{\edges,\tdisc}\, \varphi)_\edge^n =
\begin{cases} \displaystyle
\frac{|\edge|}{|D_\edge|}\ (\varphi_L^n-\varphi_K^n)\, \bfn_{K,\edge} & \mbox{ if }\edge\in\edgesint, \edge=K|L,
\\[2ex] \displaystyle
0 & \mbox{ if }\edge\in\edgesext.
\end{cases}
\end{multline}
where, for $K \in \mesh$ and $0\leq n < N$, $\varphi^n_K=\varphi(\bfx_K,t_n)$, and for a given set $A$, $\charac_A$ is the characteristic function of $A$, that is $\charac_A (x) = 1$ if $x \in A$ and $\charac_A (x) = 0$ otherwise.
Then the following weak convergence result holds.

\medskip
\begin{lemma}\label{lem:gradT}
Let $(\mesh\exm)_{m \in \xN}$ be a sequence of meshes such that the mesh step $h_{\mesh\exm}$ tends to zero when $m$ tends to $+\infty$ and satisfies the assumption \eqref{hyp:mesh_grad}.
For $m \in \xN$, we suppose given a time discretization $\tdisc\exm$, and suppose that $\delta t_{\tdisc \exm}$ also tends to zero when $m$ tends to $+\infty$.

Let $\varphi \in C_c^\infty(\Omega \times [0,T])$ and, for $m \in \xN$, $\gradi_{\edges\exm,\tdisc\exm}\, \varphi \in L^\infty(\Omega \times[0,T])^d$ be defined by \eqref{eq:gradT}.
Then the sequence $(\gradi_{\edges\exm,\tdisc\exm}\, \varphi)_{m \in \xN}$ is bounded in $L^\infty(\Omega \times[0,T])^d$ uniformly with respect to $m$ and converges to $\gradi \varphi$ in $L^\infty(\Omega \times[0,T])^d$ weak $\star$.
\end{lemma}

\medskip
\begin{proof}
Let $\bfpsi \in C_c^\infty(\Omega\times[0,T])^d$, and let us define, for $m \in \xN$, the following functions of time:
\[
I(t) = \int_\Omega \gradi \varphi(\bfx,t)\cdot \bfpsi(\bfx,t) \dx,\quad
I\exm(t) =  \int_\Omega \gradi_{\edges\exm,\tdisc\exm}\, \varphi(\bfx,t)\cdot \bfpsi(\bfx,t) \dx.
\]
Thanks to remark \ref{rem:cvu} and the regularity of $\varphi$ and $\bfpsi$, $I\exm(t)$ converges to $I(t)$ uniformly with respect to $t$, with $I \in C^\infty_c(0,T)$.
The integral
\[
\int_0^T \int_\Omega (\gradi_{\edges\exm,\tdisc\exm}\, \varphi(\bfx,t)\cdot \bfpsi(\bfx,t) \dx \dt
=\sum_{n=0}^{N\exm-1} (t_{n+1}-t_n)\ I\exm(t_n) +R\exm,
\]
where $R\exm$ is a remainder term tending to zero as $\delta t_{\tdisc \exm}$ thanks to the regularity of $\bfpsi$.
We have
\begin{multline*}
\sum_{n=0}^{N\exm-1} (t_{n+1}-t_n)\ I\exm(t_n)=\sum_{n=0}^{N\exm-1} (t_{n+1}-t_n)\ I(t_n)
\\
+ \sum_{n=0}^{N\exm-1} (t_{n+1}-t_n)\ \bigl(I\exm(t_n)-I(t_n)\bigr)
\end{multline*}
The second term tends to zero when $m$ tends to $+\infty$ thanks to the uniform convergence of $I\exm$ to $I$.
By the regularity of $I$, the first one converges to
\[
\int_0^T I(t) \dt = \int_0^T \int_\Omega \gradi \varphi(\bfx,t)\cdot \bfpsi(\bfx,t) \dx \dt.
\]
The conclusion follows by density of $C_c^\infty(\Omega\times[0,T])^d$ in $L^1(\Omega\times[0,T])^d$.\\
\end{proof}
%
%
\section{Convergence of ``discrete translations" of functions of $L^1$}\label{sec:kolm}

The proof of the original Lax-Wendroff theorem (extended to non uniform meshes in \cite[Theorem 21.2]{egh-00-fin}) relies on the mean continuity of integrable functions, which is used to prove that for a sequence $(u_p)_{p\in \xN}$ of converging functions in $L^1(\Omega)$, 
\begin{equation}\label{lim-translates}
\int_\Omega |u_p(\cdot + h) - u_p| \dx \to 0 \mbox{ as } h \to 0 \mbox{ uniformly w.r.t. } p.
\end{equation}
This proof may be extended to Cartesian meshes; however, on  an unstructured mesh, the notion of translation is no longer clear and the problem must be reformulated as follows.
For $u \in L^1(\Omega)$ and $K \in \mesh$, we denote by $u_K$ the mean value of $u$ and we set for $\edge \in \edgesint$, $\edge=K|L$, $[u]_\edge = |u_K-u_L|$.
We now introduce the quantity
\begin{equation}\label{trans-u}
T_\mesh u= \sum_{\edge \in \edgesint, \edge=K|L} |D_\edge| [u]_\edge.
\end{equation}
The objective of this section is to prove that, for a sequence $(u_p)_{p \in \xN}$ of functions of $L^1(\Omega)$ and a sequence of increasingly refined meshes $(\mesh\exm)_{m \in \xN}$, supposed to be regular in a sense to be defined, the quantity $T_{\mesh\exm} u_p$ tends to zero as $m \to +\infty$ uniformly with respect to $p$.
Note that in the case of a uniform 1D mesh such that $|D_\edge|=h$ for any $\edge \in \edges$, this is equivalent to showing \eqref{lim-translates}.

\medskip
The proof of this result is split in two steps; we first consider a fixed function $u \in L^1(\Omega)$ and prove a generalisation of the mean continuity in Lemma \ref{lem-trans};  we then address the case of a converging sequence in $L^1(\Omega)$ in Lemma \ref{lem:transuT}.
The technique used here is reminiscent of underlying arguments invoked in \cite[Section 4.2.2]{god-96-num} for two-dimensional triangular meshes; we consider here general meshes, paying a special attention to mesh regularity requirements.

\medskip
To formulate the regularity assumption of the sequence of meshes considered in this section, we introduce the following parameter:
\begin{equation} \label{reg-m}
\theta_\mesh=\max_{K \in \mesh} \max_{\edge \in \edges_K} \frac { |D_\edge|}{|K|}.
\end{equation}
By definition, we thus have $|D_\edge| \le \theta_\mesh\, |K|$, for $\edge \in \edges_K$, $K \in \mesh$.

\medskip
\begin{lemma}\label{lem-trans}
Let $\theta >0$  and $(\mesh\exm)_{m \in \xN}$ be a sequence of meshes such that $\theta_{\mesh\exm} \le \theta$ for all $m \in \xN$ and $\lim_{m \to +\infty} h_{\mesh\exm} =0$.
We suppose that the number of faces of a cell $K \in \mesh\exm$ is bounded by $\mathcal N_\edges$, for any $m \in \xN$.
Let $u \in L^1(\Omega)$, let $u_K$ denote the mean value of $u$ on a cell $K$, and let $T_{\mesh\exm} u$ be defined by \eqref{trans-u}.
Then,
\begin{equation} \label{lim-tu}
\lim_{m \to +\infty} T_{\mesh\exm} u=0.
\end{equation}
\end{lemma}
\begin{proof}
We first note that for $u \in L^1(\Omega)$ and a mesh $\mesh$,
\begin{align*}
T_\mesh(u) \le  \sum_{\edge \in \edgesint, \edge=K|L} |D_\edge|\ \bigl([u_K|+|u_L|\bigr)
&
\le \theta \sum_{\edge \in \edgesint, \edge=K|L} \bigl(|K|\ |u_K|+|L|\ |u_L|\bigr)
\\ &
\le \mathcal N_\edges\, \theta\ ||u||_{L^1(\Omega)}.
\end{align*}
Then, for any $\varphi \in C_c^\infty(\xR^d,\xR)$,
\begin{equation} \label{ine-cruc12}
T_{\mesh\exm} u \leq T_{\mesh\exm} (u-\varphi) + T_{\mesh\exm} \varphi \le  \mathcal N_\edges\, \theta\ ||u-\varphi||_{L^1(\Omega)}+ T_{\mesh\exm} \varphi.
\end{equation}
Let $\epsilon>0$.
Since $C_c^\infty(\Omega)$ is dense in $L^1(\Omega)$, there exists $\varphi \in C_c^\infty(\Omega)$ such that $\mathcal N_\edges\, \theta\ ||u-\varphi||_{L^1(\Omega)} \le \epsilon$.
Then, with this choice of $\varphi$,
\begin{equation} \label{ine-epsphi}
T_{\mesh\exm} u \le  \epsilon + T_{\mesh\exm} \varphi.
\end{equation}
We now use only the fact that $\varphi$ is Lipschitz continuous.
There exists $M_\varphi>0$ such that $|\varphi(\bfx) - \varphi(\bfy)| \le M_\varphi\, |\bfx-\bfy|$ for all $\bfx,\, \bfy \in \Omega$.
Then, for any $\edge =K|L \in \edgesint$, using $\bar K \cap \bar L \ne \emptyset$,
\begin{align*}
|K||L| [\varphi]_\edge=|K|\ |L|\ |\varphi_K -\varphi_L|
&
=\Bigl| \int_K \int_L \varphi(\bfx)-\varphi(\bfy) \dy \dx \bigr|
\\ &
\le M_\varphi \int_K \int_L |\bfx-\bfy| \dx \dy
\le M_\varphi\ |K|\ |L|\ (h_K+h_L).
\end{align*}
This yields
\[
T_{\mesh\exm} \varphi=\sum_{\edge \in \edgesint, \edge=K|L} |D_\edge| [\varphi]_\edge
\le 2 M_\varphi\, h_{\mesh\exm}\ \sum_{\edge \in \edgesint, \edge=K|L} |D_\edge|
\le 2 M_\varphi\, h_{\mesh\exm} |\Omega|.
\]
Since $\lim_{m \to +\infty} h_{\mesh\exm} =0$, there exists $m_0$ such that $2 M_\varphi\, h_{\mesh\exm} |\Omega| \le \epsilon$ for $m \ge m_0$ and then, with \eqref{ine-epsphi},
\[
T_{\mesh\exm} u \le 2 \epsilon \quad \textrm{for } m \ge m_0.
\]
\end{proof}

\medskip
The convergence result which constitutes the aim of this section is stated in the following lemma.

\medskip
\begin{lemma}\label{lem-transu}
Let $\theta >0$  and $(\mesh\exm)_{m \in \xN}$ be a sequence of meshes such that $\theta_{\mesh\exm} \le \theta$ for all $m \in \xN$ and $\lim_{m \to +\infty} h_{\mesh\exm} =0$.
We suppose that the number of faces of a cell $K \in \mesh\exm$ is bounded by $\mathcal N_\edges$, for any $m \in \xN$.
Let $u \in L^1(\Omega)$ and $(u_p)_{p \in \xN}$ be a sequence of functions of $L^1(\Omega)$ such that $u_p \to u$ in $L^1(\Omega)$ as $p \to +\infty$.
Let $T_{\mesh\exm} u_p$ be defined by \eqref{trans-u}.
\\[0.5ex]
Then $T_{\mesh\exm} u_p$ tends to zero when $m$ tends to $+\infty$ uniformly with respect to $p \in \xN$.
\end{lemma}

\medskip
\begin{proof}
Using \eqref{ine-cruc12} with $u=u_p$ and $\varphi=u$, we obtain
\[
T_{\mesh\exm} u_p \le T_{\mesh\exm} (u_p-u) + T_{\mesh\exm} u \le  \mathcal N_\edges\, \theta\ ||u-u_p||_{L^1(\Omega)}+ T_{\mesh\exm} u.
\]
Let $\epsilon >0$.
Since $u_p \to u$ in $L^1(\Omega)$, as $p \to +\infty$, there exists $p_0$ such that for $p \ge p_0$, $\mathcal N_\edges\,\theta \ ||u-u_p||_{L^1(\Omega)} \le \epsilon$ and then
\[
T_{\mesh\exm} u_p \le \epsilon + T_{\mesh\exm} u.
\]
We use now Lemma \ref{lem-trans}.
There exists $m_0$ such that $T_{\mesh\exm} u \le \epsilon$ for $m \ge m_0$ and then, for $m\ge m_0$ and $p \ge p_0$,
\[
T_{\mesh\exm} u_p \le 2 \epsilon.
\]
Using again Lemma \ref{lem-trans} for $p \in \{0,\ldots,p_0-1\}$ gives $m_1$ such for $m \ge m_1$ and $p \in \xN$,
\[
T_{\mesh\exm} u_p \le 2 \epsilon.
\]
\end{proof}

\medskip
\begin{remark}
The underlying ideas of the proofs of this section may be summed up as follows.
Let $(T\exm)_{m\in \xN}$ be a sequence of semi-norms acting on a Banach space $\mathcal B$ to $\xR^+$ satisfying two properties:
\bli
\item uniform boundedness: there exists $C >0$ such that $T\exm u \leq C\ ||u||_{\mathcal B}$ for any $u \in \mathcal B$ and $m \in \xN$,
\item convergence to zero on a dense subspace: there exists $\mathcal D \subset \mathcal B$ and dense in $\mathcal B$ such that, for any $\bar u \in \mathcal D$, $\lim_{m \to +\infty} T\exm \bar u=0$.
\end{list}
Then, for any converging sequence $(u_p)_{p\in\xN}$ in $\mathcal B$, $T\exm u_p$ tends to zero uniformly with respect to $p$.
\end{remark}

\medskip
This result extends to time depending functions as follows.
Let $T>0$ and $\tdisc$ be a time discretization of the interval $(0,T)$, as defined in the previous section.
For $u \in L^1(\Omega \times (0,T))$, $K \in \mesh$ and $0 \leq n < N$, let $u_K^{n+1}$ be the mean value of $u$ over $K \times (t_n,t_{n+1})$.
For $\edge \in \edgesint$, $\edge=K|L$ and for $0 < n \leq N$, let $[u^n]_\edge=|u_K^n-u_L^n|$; for $K \in \mesh$ and $0 < n < N$, we set $[u_K]^n=|u_K^{n+1}-u_K^n|$.
We define the following quantity:
\begin{equation}\label{transT-u}
T_{\mesh,\tdisc} u= \sum_{n=0}^{N-1} (t_{n+1}-t_n) \sum_{\edge \in \edgesint, \edge=K|L} |D_\edge| [u^{n+1}]_\edge
+ \sum_{n=1}^{N-1} (t_{n+1}-t_n) \sum_{K\in\mesh} |K| [u_K]^n.
\end{equation}
Then the following lemma results from easy extensions (in fact, reasoning in $\xR^{d+1}$ instead of $\xR^d$) of the previous proofs of this section. 

\medskip
\begin{lemma}\label{lem:transuT}
Let $\theta >0$  and $(\mesh\exm)_{m \in \xN}$ be a sequence of meshes such that $\theta_{\mesh\exm} \le \theta$ for all $m \in \xN$ and $\lim_{m \to +\infty} h_{\mesh\exm} =0$.
We suppose that the number of faces of a cell $K \in \mesh\exm$ is bounded by $\mathcal N_\edges$, for any $m \in \xN$.
For $m \in \xN$, we suppose given a time discretization $\tdisc\exm$, and suppose that $\delta t_{\tdisc \exm}$ also tends to zero when $m$ tends to $+\infty$.
Let $u \in L^1(\Omega\times(0,T))$ and $(u_p)_{p \in \xN}$ be a sequence of functions of $L^1(\Omega\times(0,T))$ such that $u_p \to u$ in $L^1(\Omega\times(0,T))$ as $p \to +\infty$.\\[0.5ex]
Then $T_{\mesh\exm,\tdisc\exm} u_p$ tends to zero when $m$ tends to $+\infty$ uniformly with respect to $p \in \xN$.
\end{lemma}
Note that the results of this section allow to show the weak consistency of conservative finite volume schemes without the $BV$ boundedness assumptions that are found in e.g. \cite[chapter 12]{lev-02-fin}, \cite{shi-18-loc}; indeed, the limit on the translates \eqref{lim-translates} is shown directly from the $L^1$ convergence assumption thanks to the above lemmas. 
In fact, finite volume schemes are known to be non TVD on non-structured meshes, so that these boundedness assumptions are difficult to satisfy in this case.
The two above lemmas are used in Theorem \ref{th:lw} below to show the weak consistency of finite volume schemes, so that the sequences of functions which are considered are piecewise constant. 
But in fact, these lemmas could also be used for other types of functions for instance in the case of higher order schemes. 

%
\section{Weak consistency of conservative finite volumes discretizations of conservation laws}\label{sec:lw}

Let us consider the following conservation law posed over $\Omega \times (0,T)$:
\begin{equation}\label{eq:claw}
\partial_t(u) + \dive(\bfF(u)) = 0,
\end{equation}
where $\bfF:\ \xR \rightarrow \xR^d$ is a regular flux function.
This equation is complemented with the initial condition $u(\bfx,0)=u_0(\bfx)$ for a.e. $\bfx \in \Omega$, where $u_0$ is a given function of $L^1(\Omega)$, and convenient boundary conditions on the boundary $\partial \Omega$, see remark \ref{rem-bc}.
For a given mesh $\mesh$ and time discretization $\tdisc$, a finite volume approximation of Equation \eqref{eq:claw} reads
\begin{equation}\label{eq:claw_disc}
\frac{|K|}{t_{n+1}-t_n}\ (u_K^{n+1}-u_K^n) + \sum_{\edge\in\edges(K)} |\edge|\ \bfF_\edge^n \cdot \bfn_{K,\edge} = 0,
\mbox{ for } K \in \mesh \mbox{ and } 0 \leq n < N,
\end{equation}
where $\bfF_\edge^n$ is the numerical flux.
Equation \eqref{eq:claw_disc} is complemented by the initial condition:
\begin{equation}\label{eq:claw_ini}
u_K^0 = \frac 1 {|K|}\ \int_K u_0(\bfx) \dx, \mbox{ for } K \in \mesh,
\end{equation}
and convenient boundary fluxes.
The discrete unknowns $(u_K^n)_{K \in \mesh, 0\leq n \leq N}$ are associated to a function of $L^\infty(\Omega)$ as follows:
\begin{equation}\label{eq:disc_func}
u(\bfx,t)=\sum_{n=0}^{N-1} \sum_{K\in\mesh} u_K^n \ \mathcal X_K\ \mathcal X_{(t_n,t_{n+1}]},
\end{equation}
where $\mathcal X_K$ and $\mathcal X_{(t_n,t_{n+1}]}$ stand for the characteristic function of $K$ and the interval $(t_n,t_{n+1}]$ respectively.

\medskip
We now suppose given a sequence of meshes $(\mesh\exm)_{m \in \xN}$ and time discretizations $(\tdisc\exm)_{m \in \xN}$, with $h_{\mesh\exm}$ and $\delta t_{\tdisc\exm}$ tending to zero as $m$ tends to $+\infty$, and, for $m \in \xN$, denote by $u\exm$ the function obtained with $\mesh\exm$ and $\tdisc\exm$.
Let us suppose that the sequence $(u\exm)_{m \in \xN}$ converges in $L^1(\Omega\times(0,T))$ to a function $\bar u$.
The aim of this section is to determine sufficient conditions for $\bar u$ to satisfy the weak formulation of Equation \eqref{eq:claw}, \ie
\begin{multline}\label{eq:claw_weak}
-\int_0^T \int_\Omega \Bigl( \bar u\ \partial_t \varphi + \bfF(\bar u) \cdot \gradi \varphi \Bigr) \dx \dt = \int_\Omega u_0(\bfx) \varphi(\bfx,0) \dx,
\\
\mbox{ for any } \varphi \in C_c^\infty(\Omega \times [0,T)).
\end{multline}
This is obtained by letting the space and time step tend to zero and passing to the limit in the scheme, and this issue is referred to as a weak consistency property of the scheme.

\medskip
\begin{remark}[Weak formulation and boundary conditions] \label{rem-bc}
Note that the weak formulation \eqref{eq:claw_weak} is incomplete, in the sense that it does not imply anything on boundary conditions, since test functions are supposed to have a support compact in $\Omega \times [0,T)$; in fact, weak formulation of boundary conditions is a difficult problem for hyperbolic problems, still essentially open for systems, and out of scope of the present paper.
\end{remark}

\begin{theorem}[An extension of the Lax-Wendroff theorem] \label{th:lw}
Let $(\mesh\exm)_{m \in \xN}$ and time discretizations $(\tdisc\exm)_{m \in \xN}$, with $h_{\mesh\exm}$ and $\delta t_{\tdisc\exm}$ tending to zero as $m$ tends to $+\infty$, and, for $m \in \xN$, denote by $u\exm$ the function obtained with $\mesh\exm$ and $\tdisc\exm$.
Assume that the sequence $(u\exm)_{m \in \xN}$ converges to in $\bar u$ $L^1(\Omega\times(0,T))$ to a function $\bar u$.
Assume furthermore that
\begin{itemize}
	\item[]$(i)$ the sequence $(\bfF(u\exm))_{m \in \xN}$ converges in $L^1(\Omega\times(0,T))$ to $\bfF(\bar u)$, 
	\item[]$(ii)$ the sequence of meshes satisfies the regularity assumptions of Lemma \ref{lem:gradT}
	\item[]$(iii)$ there exists a real number $C_F$ depending only on $F$, such that
			\begin{equation} \label{eq:hypF}
				|\bfF_\edge^n - \bfF(u_K^n)| \leq C_F\ |u_K^n-u_L^n|, \quad |\bfF_\edge^n - \bfF(u_L^n)| \leq C_F\ |u_K^n-u_L^n|,
			\end{equation}
	\item[]$(iv)$ the regularity assumptions for the mesh of Lemma \ref{lem:transuT} are satisfied.
\end{itemize}
Then $\bar u$ satisfies \eqref{eq:claw_weak} {\it i.e.} it is a weak solution of Equation \eqref{eq:claw}. 
\end{theorem}

\begin{proof}
To this purpose, let $\varphi \in C_c^\infty(\Omega \times [0,T))$;  then 
$\varphi(x,t) =0$ if the distance ${\rm dist}(x,\partial \Omega)$ from $x$ to the boundary $\partial \Omega$ is such that ${\rm dist}(x,\partial \Omega) \le {\rm dist}(\supp\varphi,\partial Q)$, and similarly $\varphi(x,t) =0$ if $|T -t | \le {\rm dist}(\supp\varphi,\partial Q),$ where $Q = \Omega \times (0,T)$.
For a given mesh $\mesh\exm$ whose size $h^{(m)}$ is such that $h^{(m)} \le {\rm dist}(\supp\varphi,\partial Q)$ and time discretisation $(t_n)_{0\leq n \leq N\exm}$ such that $t^{N^{(m)}} - t^{N^{(m)}} \le {\rm dist}(\supp\varphi,\partial Q)$, let us define $\varphi_K^n$ by
\[
\varphi_K^n = \varphi(\bfx_K,t_n), \mbox{ for } K \in \mesh\exm \mbox{ and } 0 \leq n \leq N\exm,
\]
where $\bfx_K$ stands for a point of $K$.
Let us multiply Equation \eqref{eq:claw_disc} by $(t_{n+1}-t_n)\,\varphi_K^n$ and sum over the cells and time steps, to obtain $T_1\exm + T_2\exm =0$, with
\begin{align*} &
T_1\exm =\sum_{n=0}^{N\exm-1} \sum_{K\in\mesh\exm} |K|\ (u_K^{n+1}-u_K^n)\ \varphi_K^n,
\\ &
T_2\exm= \sum_{n=0}^{N\exm-1} (t_{n+1}-t_n) \sum_{K\in\mesh\exm} \varphi_K^n \sum_{\edge=K|L} |\edge|\ \bfF_\edge^n \cdot \bfn_{K,\edge},
\end{align*}
where the notation $\sum_{\edge=K|L}$ means that the summation is performed over the internal faces of the cell $K$, each internal face separating $K$ from an adjacent cell denoted by $L$.
Reordering the sums (this may be seen as a discrete integration by parts with respect to time), the term $T_1\exm$ may be recast as $T_1\exm = \tilde T_{1,1}\exm + \tilde T_{1,2}\exm+R_1\exm$, with
\begin{align*} &
\tilde T_{1,1}\exm= - \sum_{n=0}^{N\exm-1} \sum_{K\in\mesh\exm} (t_{n+1}-t_n)\ |K|\ \frac{\varphi_K^{n+1}-\varphi_K^n}{t_{n+1}-t_n}\ u_K^n,
\\[1.5ex] &
\tilde T_{1,2}\exm= - \sum_{K\in\mesh\exm} |K|\ u_K^0\ \varphi_K^0,
\\[1.5ex] &
R_1\exm=- \sum_{n=0}^{N\exm-1} \sum_{K\in\mesh\exm} (t_{n+1}-t_n)\ |K|\ \frac{\varphi_K^{n+1}-\varphi_K^n}{t_{n+1}-t_n}\ (u_K^{n+1}-u_K^n).
\end{align*}
Let us denote by $\eth_t\exm \varphi$ the following time discrete derivative of $\varphi$:
\[
\eth_t\exm \varphi = \sum_{n=0}^{N\exm-1} \sum_{K\in\mesh\exm} \frac{\varphi_K^{n+1}-\varphi_K^n}{t_{n+1}-t_n}\ \mathcal X_K\ \mathcal X_{(t_n,t_{n+1}]}.
\]
Thanks to the regularity of $\varphi$, $\eth_t\exm \varphi$ tends to $\partial_t \varphi$ when $m$ tends to $+\infty$ in $L^\infty(\Omega\times(0,T))$.
Since
\[
\tilde T_{1,1}\exm= -\int_0^T \int_\Omega u\exm\ \eth_t\exm \varphi \dx \dt,
\]
we thus get
\[
\lim_{m \to +\infty} \tilde T_{1,1}\exm = -\int_0^T \int_\Omega \bar u\ \partial_t \varphi \dx \dt.
\]
Thanks to Equation \eqref{eq:claw_ini} and the regularity of $\varphi$, we also get
\[
\lim_{m \to +\infty} \tilde T_{1,2}\exm = - \int_\Omega u_0(\bfx)\, \varphi(\bfx,0) \dx.
\]
Finally,
\[
|R_1\exm| \leq C_\varphi\ \sum_{n=0}^{N\exm-1} (t_{n+1}-t_n) \sum_{K\in\mesh\exm} |K|\ [u_K^{n+1}-u_K^n]
\]
with $C_\varphi$ the Lipschitz-continuity constant of $\varphi$, and $R_1\exm$ tends to zero thanks to Lemma \ref{lem:transuT}.

\medskip 
Let us now turn to term $T_2\exm$.
Reordering the sums (which, now, may be seen as a discrete integration by part with respect to space), we get:
\[
T_2\exm = -\sum_{n=0}^{N\exm-1} (t_{n+1}-t_n) \sum_{\edge \in \edgesint\exm,\,\edge=K|L} |\edge|\ \bfF_\edge^n \cdot \bfn_{K,\edge} (\varphi_L^n-\varphi_K^n).
\]
For $\edge \in \edgesint$, $\edge=K|L$, we must now define a volume $D_\edge$, and this choice may be to some extent tuned according to the mesh at hand (see Remark \ref{rem:D}).
Let us for instance suppose here that $D_\edge = D_{K,\edge} \cup D_{L,\edge}$, with $D_{K,\edge}=K\cap D_\edge $ (respectively $D_{L,\edge}=L\cap D_\edge$) and $|D_{K,\edge}| = |K|/\mathcal N_K$, where $\mathcal N_K$ stands for the number of edges of $K$ (it is easy to check that such a partition exists, noting that $D_{K,\edge}$ needs not be a polyhedron).
With this definition, we may write $T_2\exm =  \tilde T_2\exm + R\exm$ with
\begin{multline*}
\tilde T_2\exm=
-\sum_{n=0}^{N\exm-1} (t_{n+1}-t_n) 
\\
\sum_{\edge \in \edgesint\exm,\,\edge=K|L}
\Bigl( |D_{K,\edge}|\ \bfF(u_K^n) + |D_{L,\edge}|\ \bfF(u_L^n) \Bigr)  \cdot
\Bigl( \frac{|\edge|}{|D_\edge|}\ (\varphi_L^n-\varphi_K^n)\,\bfn_{K,\edge} \Bigr).
\end{multline*}
We identify
\[
\tilde T_2\exm = \int_0^T \int_\Omega \bfF(u\exm) \cdot \gradi_{\edges\exm,\tdisc\exm} \varphi \dx \dt,
\]
with the definition \eqref{eq:gradT} of $\gradi_{\edges,\tdisc} \varphi$.
Thanks to the assumptions $(i)$ and $(ii)$ of the theorem, we get that
\[
\lim_{m \to +\infty} T_1\exm = -\int_0^T \int_\Omega \bfF(\bar u) \cdot \gradi \varphi \dx \dt.
\]
Thanks to the assumption $(iii)$ and owing to the regularity of $\varphi$ which implies that $\gradi_{\edges,\tdisc} \varphi$ is bounded in $L^\infty(\Omega\times(0,T))^d$ under the mesh regularity assumptions of Lemma \ref{lem:gradT}, we obtain that
\[
|R\exm| \leq C_F\ C_\varphi\ \sum_{n=0}^{N\exm-1} (t_{n+1}-t_n) \sum_{\edge \in \edgesint\exm,\,\edge=K|L} |D_\edge|\ [u\exm]_\edge,
\]
where $C_\varphi \in \xR_+$ depends only on $\varphi$, and thus, thanks to the assumption  $(iv)$, $R\exm$ tends to zero as $M$ tends to $+\infty$.
\end{proof}

Let us now comment on the assumptions used in the theorem
\bli
\item Assumptions $(ii)$ and $(iv)$ are regularity assumptions for the sequence of meshes.
Hypothesis $(iv)$ amounts to suppose that the number of cell faces is bounded and, with the definition of the volumes $D_\edge$ chosen in this section, that the ratio $|K|/|L|$, for any pair $(K,L)$ of neighbouring cells, is bounded.
The constraints associated to Assumption $(ii)$ are discussed in Remark \ref{rem:reg_grad} (for instance, we may assume that the cells do not becomes too flat, in the sense of Inequality \eqref{eq:nflat}).
\item Assumption $(i)$ states the convergence of the sequence $(\bfF(u\exm))_{m \in \xN}$ to $\bfF(\bar u)$ in $L^1(\Omega\times(0,T))$.
It is just a consequence of the convergence of $(u\exm)_{m \in \xN}$ to $\bar u$ in $L^1(\Omega\times(0,T))$ if the function $\bfF$ is bounded (thanks to the Lebesgue dominated convergence theorem); in the other cases, it demands stronger convergence assumptions (for instance, with $F(u)=u^2$, convergence in $L^2(\Omega\times(0,T))$ is required).
\item Assumption $(iii)$ (\ie\ Inequalities \eqref{eq:hypF}) is a constraint over the numerical flux.
For instance, with a two-point flux $\bfF_\edge=\bfF_\edge(u_K,u_L)$, it is implied by the usual assumptions:
\[
\bfF_\edge(u,u)=\bfF(u)\ \forall u \in \xR,\ \bfF_\edge \mbox{ Lipschitz continuous w.r.t. its arguments}.
\]

\end{list}

\begin{remark}[On the Lipschitz assumption on the flux]\label{rem-flux}
The Lipschitz continuity assumption on the flux may in fact be replaced by the following weaker ``Lip-diag" condition:
\begin{align}
	&|F(a,b) -F(a,a)| \le C_F |a-b|\\
	&|F(b,a) -F(a,a)| \le C_F |a-b|.
	\label{lip-diag}
\end{align}

In the case of a MUSCL scheme for the convection equation by a regular velocity $\bfb$, the discretization of the flux on $\edge=K|L$ is given by $u_\edge\,\bfb\cdot \bfn_{K,\edge}$ with $u_\edge$ depending on $u_K$ and $u_L$ but also on another upwind cell; however, if the flux is assumed to be Lip-diag with respect to the value $u_K$,  $u_\edge$ is always a convex combination of $u_K$ and $u_L$ and Assumption $(iii)$ holds.

It seems impossible to relax this assumption to the case of a merely continuous flux without additional conditions: indeed in \cite{ell-07-lax}, a counterexample is given for non space-time quasi-uniform grids; in this counter-example, the space mesh is quasi-uniform, so that the assumptions of Theorem \ref{th:lw} are clearly satisfied apart from the Lipschitz assumption on the flux. 

Moreover, the Lip-diag framework seems interesting for a number of schemes; for instance when replacing the Roe flux by the Rusanov flux when an entropy correction is needed, the continuity of the flux is lost but not the Lip-diag property. 

\end{remark}

\medskip
Finally, let us conclude by mentioning a possible generalization of $(iii)$ for multiple point fluxes.
For instance, if, on $\edge=K|L$, $\bfF_\edge=\bfF_\edge(u_K,u_L,u_M)$ with $M$ a neighbour of $L$, we may replace the first inequality of \eqref{eq:hypF} by
\[
|\bfF_\edge^n - \bfF(u_K^n)| \leq C_F\ \Bigl(|u_K^n-u_L^n| + |u_K^n-u_M^n|\bigr),
\]
and then use $|u_K^n-u_M^n| \leq |u_K^n-u_L^n|+|u_L^n-u_M^n|$.
We obtain that $R\exm$ still reads as a summation of jumps across the faces; however, the weight of $[u]_\edge$ is now more complex, and its control by $|D_\edge|$ needs a stronger regularity assumption on the mesh.
Such a situation is faced, for instance, when using a second order Runge-Kutta scheme for the time discretization instead of the Euler scheme.
%
%
\bibliographystyle{plain}
\bibliography{lw} 
\end{document}